\documentclass[review]{elsarticle}

\usepackage{lineno,hyperref}
\modulolinenumbers[5]

\usepackage{amsfonts}
\usepackage{amsmath, amssymb}
\usepackage{pifont}
\usepackage{natbib}
\usepackage{color}

\usepackage{graphicx}
\usepackage{amsthm}
 
\newtheorem{defn}{Definition}
\numberwithin{defn}{section} 

\newtheorem{theorem}{Theorem}
\numberwithin{theorem}{section}
\newtheorem{propn}{Proposition}
\numberwithin{propn}{section}

\numberwithin{conjecture}{section}

\usepackage{amsthm}

\newtheoremstyle{dotless}{}{}{\itshape}{}{\bfseries}{}{ }{}

\theoremstyle{dotless}
\newtheorem{stheorem}[theorem]{Theorem}
\usepackage{enumerate}

\usepackage{mathtools}

\usepackage{url}

\journal{.}

\begin{document}

%%% my macros
\def\R{\mathbb{R}}
\def\P{\mathbb{P}}
\def\Z{\mathbb{Z}}
\def\N{\mathbb{N}}
\def\Q{\mathbb{Q}}
\def\D{\mathbb{D}}
\def\Sp{{\mathbb{S}}}
\def\T{\mathbb{T}}
\def\H{\mathbb{H}}
\def\hb{\hfil \break}
\def\ni{\noindent}
\def\i{\indent}
\def\a{\alpha}
\def\b{\beta}
\def\e{\epsilon}
\def\d{\delta}
\def\De{\Delta}
\def\g{\gamma}
\def\qq{\qquad}
\def\L{\Lambda}
\def\E{\cal E}
\def\G{\Gamma}
\def\F{\cal F}
\def\K{\cal K}
\def\A{\cal A}
\def\B{\cal B}
\def\M{\cal M}
\def\P{\cal P}
\def\Om{\Omega}
\def\om{\omega}
\def\s{\sigma}
\def\th{\theta}
\def\Th{\Theta}
\def\vp{\varphi}
\def\z{\zeta}
\def\p{\phi}
\def\m{\mu}
\def\n{\nu}
\def\l{\lambda}
\def\Si{\Sigma}
\def\q{\quad}
\def\qq{\qquad}
\def\half{\frac{1}{2}}
\def\hb{\hfil \break}
\def\half{\frac{1}{2}}
\def\pa{\partial}
\def\r{\rho}
\def\Spd{{{\Sp}}^d}
\def\mont{\left({\mathcal{I}}f\right)}
\def\desc{\left({\mathcal{D}}f\right)}
\def\curlp{{\mathcal{P}}(\Sp^d)}
\def\curlpr{{\mathcal{P}}(\Sp^d \times \R)}
\def\du{\, \mathrm{d}u}
%%%%

\begin{frontmatter}

\title{Dimension walks on $\Sp^d \times \R$.}

\author{N. H. Bingham}
\author{Tasmin L. Symons\corref{mycorrespondingauthor}}
\address{Department of Mathematics, \\ Imperial College London, \\ London, UK, SW7 2AZ}
\cortext[mycorrespondingauthor]{Corresponding author}
\ead{tls111@ic.ac.uk}



\begin{abstract}
We verify that the established one- and two-step recurrences for positive definite functions on spheres $\Sp^d$ extend to the spatio-temporal case of $\Sp^d \times \R$.
\end{abstract}


\end{frontmatter}

\linenumbers

\section{Introduction}

Let $X$ be a mean-zero and isotropic Gaussian process index by the unit sphere $\Sp^d$ embedded in $\R^{d+1}$, with continuous covariance 
$$
\text{Cov}(\theta_1, \theta_2) = (f \circ \cos) (d(\theta_1, \theta_2)), \qquad \theta_1, \theta_2 \in \Sp^d, f: [-1,1] \to \R,
$$
where $d(\cdot, \cdot)$ denotes geodesic distance on the sphere. 

Gaussian processes on the sphere are of particular interest in cosmology and geostatistics, with the latter application highlighted recently by Gneiting \cite{gne}. This reignited interest in the geostatistical theory of these processes has coincided neatly with the growing availability of high-resolution, global satellite data. Here the curvature of the Earth becomes essential: whilst the sphere is a manifold, thus locally Euclidean, a ``flat Earth'' approximation runs counter to spherical geometry over large distances. 

As Gaussian processes are determined entirely by their means (which we take here to be constant and zero) and covariance functions, the relevant mathematical theory is of positive definite functions on the sphere, and is long-established and rich.

\begin{defn}
The class of continuous functions $f : [-1, 1] \to \R$ such that $(f \circ \cos)$ is positive definite on $\Sp^d$ is denoted $\mathcal{P}(\Sp^d)$.
\end{defn}

The central theorem is Schoenberg's characterisation of $\mathcal{P}(\Sp^d)$ by Fourier-Gegenbauer expansions.  The ultraspherical (Gegenbauer) polynomials are 
$$
W_n^{\nu}(x) := C_n^{\nu}(x)/C_n^{\nu}(1) = C_n^{\nu}(x).n! \Gamma(2 \nu)/\Gamma(n + 2 \nu);
$$
these are orthogonal under the probability measure
$$
G_{\nu}(dx) := \frac{\G(\nu + 1)}{\sqrt{\pi} \G(\nu + \half)}.
(1 - x^2)^{\nu - \half} dx \qquad (x \in [-1,1]):
$$
$$
\int_{-1}^1 W_m^{\nu}(x) W_n^{\nu}(x) G_{\nu}(dx) = {\delta}_{mn}/{\om}_n^{\nu}, \qquad
{\om}_n^{\nu} := \frac{(n + \nu)}{\nu}.\frac{\G(n + 2 \nu)}{\G(2 \nu)}.
$$
The Fourier-Gegenbauer coefficients and series of $f \in C([-1,1])$ are
$$
\hat f(n) := \int_{-1}^1 f(x) W_n^{\nu}(x) G_{\nu}(dx); \qquad
f(x) \sim \sum_{n=0}^{\infty} {\om}_n^{\nu} \hat f (n) W_n^{\nu}(x).
$$

\begin{stheorem}{\emph{(Schoenberg's Theorem \cite{sch}).}}\label{thrm:Sch}\\
Let $d \in \N$, $\l = (d-1)/2$. A continuous function $f: [-1,1] \to \R$ is in $\mathcal{P}(\Sp^d)$ if and only if 
$$
f(x) = \sum_{n=0}^\infty a_n W^\l_n(x)
$$ 
with $(a_n)$ a probability distribution on $\N_0 = \N \cup \{ 0\}$. Then $a_n = {\om}_n^{\nu} \hat f (n)$. 
\end{stheorem}
The Fourier-Gegenbauer expansion of Schoenberg's theorem can be viewed as an infinite {\it mixture} of Gegenbauer polynomials, with mixture (weighting) law given by the (probability) sequence $(a_n)_{(n\in \N)}$. 

Often spatial data is also temporally distributed, raising the question of an analogous result to Theorem \ref{thrm:Sch} for isotropic positive definite functions defined on $\Sp^d \times \R$. The answer is provided in \cite{berp}, with an alternative proof via the Bochner-Godement theorem offered in \cite{bins}.

\begin{theorem}\label{thrm:BP}
For $d \in \N$, $\l = (d-1)/2$, a continuous function $f: [-1,1] \times \R \to \R$ is in $\mathcal{P}(\Sp^d \times \R)$ if and only if 
$$
f(x, t) = c \sum_{n=0}^\infty a_n(t) W^\l_n(x),
$$ 
with $c \in (0, \infty)$ and $(a_n(t))$      a sequence of positive definite functions on $\R$ with $\sum a_n(0) = 1$. Then $a_n(t) = {\om}_n^{\nu} \hat f(n,t)$. 
\end{theorem} 

Question 10 in the recent survey by Porcu et al. \cite{poraf} raised the possibility of `walks on dimensions' on $\Sp^d \times \R$. Walks on dimensions, a term coined by Wendland \cite{wend}, are operators which preserve positivity but not dimension, and are useful in the construction of covariance functions for applications.  See for example Gaspari and Cohn \cite{gasc}, (Example 3.7, the multiquadric family, and Ex. 3.8, the $C_2$-Wendland family).  For members of $\mathcal{P}(\Sp^d)$, walks of one- and two- steps were given by Beatson and zu Castell \cite{beaz1, beaz2}. Our purpose here is to show that their results extend straight forwardly to the spatio-temporal setting, answering the question of  Porcu et al. \cite{poraf}.

\section{Two-step walks: mont\'ee and descente}

The recent work of Beatson and zu Castell \cite{beaz1} gives dimension walks for the classes $\mathcal{P}(\Sp^d)$, using spherical analogues of Matheron's {\sl mont\'ee} and {\sl descente} operators in Euclidean space \cite{mat1, mat2}.  These arise in geological contexts such as mining.  The mont\'ee raises the degree of regularity by integrating, so decreasing the dimension; the descente does the reverse. \\  

\begin{defn}
For $f \in \mathcal{P}(\Sp^d)$ absolutely continuous, define the descente by
\begin{equation}
\left( \tilde{\mathcal{D}} f \right) (x) := f'(x) \qquad x \in [-1, 1], \tag{D}
\label{eq:desc}
\end{equation}
and the mont\'ee by
\begin{equation}
\left( \tilde{\mathcal{I}} f \right) (x) := \int_{-1}^x f(t) \, \mathrm{d}t. \tag{M}
\label{eq:monte}
\end{equation}
\end{defn}
\ni We have (\cite{sze} (4.7.14))
\begin{equation*}
\tilde{\mathcal{D}}W^\l_n = 2\mu_\l W^{\l +1}_{n-1}, 
\end{equation*}
and
\begin{equation*}
\tilde{\mathcal{I}} W^{\l + 1}_{n-1} = \frac{1}{2\mu_\l} \left( W^\l_n - W^\l_n(-1)\right),
\end{equation*}
where
$$
\mu_\l = \begin{cases}
\l, \qquad \l > 0, \\
1, \qquad \l = 0.
\end{cases}
$$
Using these properties, \cite{beaz1} showed that, given $f \in \mathcal{P}(\Sp^{d+2})$ integrable, $C + \tilde{\mathcal{I}}f \in \mathcal{P}(\Sp^d)$ for some $C \in \R$, and conversely, given $g \in \mathcal{P}(\Sp^d)$ absolutely continuous, $\tilde{\mathcal{D}}g \in \mathcal{P}(\Sp^{d+2})$. By direct analogy to their work, we obtain the following definition and theorems.

\begin{defn}
For $f \in \mathcal{P}(\Sp^d \times \R)$  differentiable in its first variable, define
$$
\desc (x,t) := \frac{\partial}{\partial x} f(x, t).
$$
For $f \in \mathcal{P}(\Sp^d \times \R)$ integrable define
$$
\mont (x,t) := \int_{-1}^x f(u, t) \du.
$$
\end{defn}

\ni By Corollary 3.9 in \cite{berp}, $\partial^n f(x, t) / \partial x^n$ exists and is continuous when $n \leq \l$. Thus we can only guarantee existence and continuity of $\desc (x, t)$ for $\l \geq 1$, i.e. $d \geq 3$. For the interesting cases $d=1$ and $d=2$ we need to impose additional differentiability assumptions, namely that $\partial f(x, t) / \partial x$ exists and is continuous, to ensure the uniform convergence of its Schoenberg expansion.

\begin{theorem}
For $d \in \N$ and $f \in \mathcal{P}(\Sp^{d+2} \times \R)$, there exists a constant $C$ such that $C + \mathcal{I}f \in \mathcal{P}(\Sp^d \times R)$.
\end{theorem}

\begin{proof}[\emph{\textbf{Proof}}]
The proof proceeds as in \cite{beaz1}, with the appropriate changes.

Since $f(x,t) \in \mathcal{P}(\Sp^{d+2} \times \R)$, it has the Schoenberg expansion
$$
f(x, t) = \sum_{n=0}^\infty a_n(t) W^{\l +1}_n(x),
$$
with $a_n(t)$ as in Theorem 1.2. Integrating term by term and using (\ref{eq:desc}):
\begin{align*}
\mont (x,t) &= \int_{-1}^x \sum_{n=0}^\infty a_n(t) W^{\l + 1}_n(u) \du \\
&= \sum_{n=0}^\infty a_n(t) \int_{-1}^x W^{\l +1}_n(u) \du \\
&= \sum_{n=0}^\infty a_n(t) \, \tilde{\mathcal{I}}W^{\l+1}_n \\
&= \sum_{n=0}^\infty a_n(t) \frac{1}{2\mu_\l} \left( W^\l_{n+1}(x) - W^\l_{n+1}(-1) \right) \\
&= \sum_{n=1}^\infty \frac{a_{n-1}(t)}{2\mu_\l} W^\l_n(x) - \sum_{n=1}^\infty \frac{a_{n-1}(t)}{2\mu_\l} W^\l_n(-1).
\end{align*}
So,
$$
(\mathcal{I}f)(x,t) = \sum_{n=0}^\infty b_n(t) W^\l_n(x),
$$
where
$$
b_n(t) = \begin{dcases}
\frac{a_{n-1}(t)}{2\mu_\l}, \qquad n = 1, 2, 3, \ldots \\
- \sum_{i=0}^\infty \frac{a_{i-1}(t)}{2\mu_\l} W^\l_i(-1), \qquad n = 0.
\end{dcases}
$$
When $n \geq 1$ the $b_n(t)$ are clearly positive definite on $\R$, and $|2 \mu_\l \sum_{n=1}^\infty b_n(t)| \leq 2 \mu_\l \sum_{n=1}^\infty b_n(0) = \sum_{n=0}^\infty a_n(0) < \infty$. The only term which needs more comment is $b_0(t)$. The $n = 0$ term in the Gegenbauer expansion, $b_0(t)W^\l_0(x) = b_0(t)$, is constant in $x$, so as $|b_0(t) | \leq b_0(0) =: C$, say, we find $C + \mathcal{I}f \in \curlpr$.

When $\l = 0$, as above,
\begin{align*}
|b_0(t)| &= \bigl\lvert \half \sum_{n=1}^\infty a_{n-1}(t)W^0_n(-1) \bigr\rvert \\
&= \bigl\lvert \half \sum_{n=1}^\infty a_{n-1}(t) \bigr\rvert \leq \half \sum_{n=1}^\infty |a_{n-t}(t) | \\
&\leq \half \sum_{n=1}^\infty a_{n-1}(0) =: C < \infty.
\end{align*}

Similarly for $\l > 0$, $W^\l_n(-x) = (-1)^nW^\l_n(x)$ (\cite{sze}, (4.1.3)), so
\begin{align*}
|b_0(t)| 
&= \bigl\lvert \sum_{m=1}^\infty \frac{(-1)^{m+1}a_{m-1}(t)}{2\mu_\l} W^\l_m(1) \bigr\rvert 
% &\leq  \sum_{m=1}^\infty \bigl\lvert\frac{(-1)^{m+1}a_{m-1}(t)}{2\mu_\l} W^\l_m(1) % \bigr\rvert  \\
\leq \sum_{m=1}^\infty \frac{|a_{m-1}(t)|}{2\mu_\l} \leq \sum_{m=1}^\infty \frac{|a_{m-1}(0)|}{2\mu_\l}.
\end{align*}
By (\cite{sze}, Th. 9.1.3) the sum is $(C,k)$-summable for $k > \l + 1/2$ and so, since all the terms in the sum are non-negative, convergent (\cite{har} Th. 64).  So
$$
|b_0(t)| 
\leq \sum_{m=1}^\infty \frac{a_{m-1}(0)}{2\mu_\l} 
=: C < \infty, 
$$
giving the required bound for $b_0(t)$.          \end{proof}

\begin{theorem}
For $d \in \N$ and $f: [-1, 1] \times \R \to \R$ continuously differentiable in its first argument: if $f \in \curlpr$, then $\partial / \partial x f(x, t) \in \mathcal{P}(\Sp^{d+2} \times \R).$
\end{theorem}

\begin{proof}[\emph{\textbf{Proof.}}]

We may differentiate termwise, as the series of derivatives is uniformly convergent:
\begin{align*}
\frac{\partial}{\partial x} f(x, t) 
&= \frac{\partial}{\partial x} \left( a_0(t) +  \sum_{n=1}^\infty a_n(t) W^\l_n(x) \right) 
= \sum_{n=1}^\infty a_n(t) \frac{\partial}{\partial x} W^\l_n(x) \\
 &= \sum_{n=1}^\infty a_n(t) 2 \mu_\l W^{\l+1}_{n-1}(x) \tag{by ($D$)} \\
 &= \sum_{n=0}^\infty 2 \mu_\l a_{n+1}(t) W^{\l+1}_n(x).
\end{align*}
Or, one may extend the proof of (\cite{beaz1}, Lemma 2.4, Th. 2.3), along the same lines as in the proof above.   \end{proof}

We also have the following result for the infinite-dimensional case (the Hilbert sphere), analogous to that in \cite{trubz}.  

\begin{theorem}
The class $\mathcal{P}(\Sp^\infty \times \R)$ of geotemporal covariances on the Hilbert sphere cross time is closed under $\mathcal{I}$ (up to an additive constant). It is not closed under $\mathcal{D}$: if $f \in \mathcal{P}(\Sp^\infty \times \R)$, then $\desc \in \mathcal{P}(\Sp^\infty \times \R)$ only if $\sum n a_n(0) < \infty$.
\end{theorem}

\begin{proof}
This is straightforward: functions $f \in \mathcal{P}(\Sp^\infty \times \R)$ have the Schoenberg expansion
$$
f(x, t) = \sum_{n=0}^\infty a_n(t) x^n.
$$
Integrating with respect to $x$:
$$
\mont (x, t) 
= \int_{-1}^x f(u, t) du
 = \sum_{n=0}^\infty \frac{a_n(t)}{n+1} \left(x^{n+1} - (-1)^{n+1}\right ) 
 = \sum_{n=0}^\infty b_n(t)x^n,
 $$
where, as in Theorem 2.1, 
$$
b_n(t) = \begin{cases}
a_{n-1}(t)/n, \qquad n = 1, 2, \ldots \\
\sum_{i=0}^\infty \left(a_i(t)(-1)^i\right)/(i+1), \qquad n = 0.
\end{cases}
$$
Also $b_0(t)$ is bounded by $\sum a_n(0) / (n+1) < \infty$. So for a suitable constant $C$, $C + \mont \in \mathcal{P}(\Sp^\infty \times \R)$.

Turning to the {\sl descente} operator: by (\cite{BerPP}, Thrm. 1.1.) the coefficients $a_n(t)$ are given by
$$
a_n(t) = \tilde{a}_n(x,t)\big|_{x=0}, \quad \tilde{a}_n(x,t) = \frac{1}{n!} \frac{\partial^n f(x,t)}{\partial x^n}.
$$
So, as in Theorem 2.2, consider
$$
\frac{\partial f(x,t)}{\partial x} = \sum_{n=0}^\infty b_n(t) x^n,
$$
where
$$
b_n(t) = \tilde{b}_n(x, t) \big|_{x=0}, \qquad \tilde{b}_n(t) = \frac{1}{n!}\frac{\partial^{n+1}f(x,t)}{\partial x^{n+1}}.
$$
So,
$$
\tilde{b}_n(x,t) = (n+1) \tilde{a}_{n+1}(x,t),
$$
and
$$
b_n(t) = (n+1)a_{n+1}(t).
$$
Thus, if $(n+1)a_{n+1}(0) < \infty$, $\partial f(x,t) / \partial x \in \mathcal{P}(\Sp^\infty \times \R)$.
\end{proof}

\section{One-Step Walks}

In \cite{beaz2} Beatson and zu Castell proposed one-step walk alternatives to the {\sl mont\'ee} and {\sl descente} operators using weighted {\sl Riemann-Liouville fractional integrals and derivatives}. As above, these results extend straightforwardly to the geotemporal case with little modification.

\begin{defn}[following {\cite{beaz2}}]
Define, for $f(x, t) \in \mathcal{P}(\Sp^d \times \R)$ and $\l \geq 0$,
\begin{align*}
{}^{\a}I^\l_+ f(x, t) &:= (1 + x)^{\l + \a} \int_{-1}^x (x - u)^{\a - 1} (1+u)^\l f(u, t) \du ,\\
{}^\a I^\l_- f(x, t) &:= (1 - x)^{\l + \a} \int_{x}^1 (u - x)^{\a - 1} (1-u)^\l f(u, t) \du ,\\
{}^{\a}{\mathcal{I}}^\l_{\pm} &:= {}^{\a}I^\l_+ \, {\pm} \, {}^{\a}I^\l_-.
\end{align*}
If $f$ is absolutely continuous then we may also define
\begin{align*}
{}^{\a}D^\l_+ f(x,t) &:= (1+x) \frac{\partial}{\partial x} \left( (1+x)^{-\l} \int_{-1}^x (x - u)^{\a - 1} (1+u)^{\l - \a} f(u, t) \du \right), \\
{}^{\a}D^\l_- f(x,t) &:= (1-x) \frac{\partial}{\partial x} \left( (1-x)^{-\l} \int_{x}^1 (u - x)^{\a - 1} (1 -u)^{\l - \a} f(u, t) \du \right), \\
{}^{\a}{\mathcal{D}}^\l_{\pm} &:= {}^{\a}D^\l_+ \, {\pm} \, {}^{\a}D^\l_-.
\end{align*}
\end{defn}

Unlike the classical Riemann-Liouville operators ${}^{\a}I^{RL}_\pm$ and ${}^{\a}D^{RL}_\pm$ the operators ${}^{\a}I^\l_\pm$ and ${}^{\a}D^\l_\pm$ do not have the semi-group property -- that is. integration/differentiation of orders ${\a}_1$ and ${\a}_2$ in succession does not yield order ${\a}_1 + {\a}_2$ in general. \\ 
\indent The case $\a = 1/2$ yields one-step walks on spheres. The actions of the half-step operators on the Gegenbauer polynomials are found in \cite{beaz2}, and summarised below, abbreviating ${}^{\half}\mathcal{I}^\l_\pm $ to $\mathcal{J}^\l_\pm$ for convenience. 

\begin{propn}\cite{beaz2}
For $\l > 0$, $n \in \N$, $x \in [-1,1]$,
\begin{align*}
\mathcal{J}^\l_+ W^{\l+1/2}_n(x) 
&= \frac{\sqrt{\pi} \Gamma(\l)}{\Gamma(\l + 1/2)} \frac{n+2\l}{n + \l + 1/2} W^\l_n (x), \\
\mathcal{J}^\l_- W^{\l+1/2}_n(x, 0) 
&= \frac{\sqrt{\pi} \Gamma(\l)}{\Gamma(\l + 1/2)} \frac{n+2\l}{n + \l + 1/2} W^\l_{n+1} (x). 
\end{align*}
When $\l = 0$ simply take the limit $\l \to 0^+$ and obtain
$$
\mathcal{J}^0_+ P_n(x) = \frac{2}{n + 1/2} T_n(x), \qquad \mathcal{J}^0_- P_n(x) = \frac{2}{n + 1/2} T_{n+1}(x).
$$
where $P_n(x)$ are the Legendre polynomials.
\end{propn}

\begin{propn}\cite{beaz2}
For $\l > 0$, $n \in \N$, $x \in [-1,1]$,
\begin{align*}
\mathcal{D}^\l_+ W^{\l}_n(x) &= \frac{\sqrt{\pi} \Gamma(\l + 1/2)}{\Gamma(\l )} \frac{n+2\l}{n + \l } W^{\l + 1/2}l_{n-1} (x), \\
\mathcal{D}^\l_- W^{\l}_n(x) &= \frac{\sqrt{\pi} \Gamma(\l + 1/2)}{\Gamma(\l)} \frac{2n}{n + \l} W^{\l +1/2}_n (x). 
\end{align*}
When $\l = 0$ again take the limit $\l \to 0^+$ and obtain
$$
\mathcal{D}^0_+ T_n(x) = n \pi P_{n-1}(x), \qquad \mathcal{D}^0_- T_n(x) = n \pi P_n(x).
$$
\end{propn}

The results below extend those of \cite{beaz2} to our context.

\begin{theorem}If $d \in \N$ and $f \in \mathcal{P}(\Sp^{d+1} \times \R)$, then $\mathcal{J}^\l_\pm f \in \mathcal{P}(\Sp^d \times \R)$. 
\end{theorem}

\begin{proof} 
As in Section 2, the proof follows Beatson and zu Castell's exposition \cite{beaz2}. The boundedness and positivity of the operations $\mathcal{J}^\l_\pm$ follows as there. The crux of $\mathcal{J}^\l_\pm$'s effect on functions in $\mathcal{P}(\Sp^{d+1}\times \R)$ is its action on the Gegenbauer polynomials, described in Theorem 3.3 of \cite{beaz2}.
\end{proof}

\begin{theorem}
If $d \in \N$ and $f \in \mathcal{P}(\Sp^d \times \R)$ and $\mathcal{D}^\l_\pm f$ are continuous, then $\mathcal{D}^\l_\pm f \in \mathcal{P}(\Sp^{d+1} \times \R)$.
\end{theorem}

\begin{proof}
We adapt the proof of the corresponding result in \cite{beaz2}.  As in the proof of Theorem 2.1, the nub is the use of the Tauberian theorem (\cite{har}, Th. 64) to pass from Ces\`aro convergence with non-negative summands to convergence.  Here matters are more complicated, as here we are dealing with series of positive definite functions, rather than of non-negative reals.  But, we may reduce this case to the previous one.  For, the definition of positive definiteness involves the non-negativity of the relevant quadratic forms, for arbitrary degree and arbitrary choice of scalars.  For each such choice, the `Ces\`aro-Tauberian' theorem above gives convergence.  Since this holds for each such choice, one has the desired extension, and the rest of the proof is as in \cite{beaz2}. 
\end{proof}

\section*{References}

\bibliography{wod}

\vspace{2em}

\footnotesize{
\noindent {\textbf{N. H. Bingham}, Mathematics Department, Imperial College London, UK, SW7 2AZ} \\
\noindent \url{n.bingham@imperial.ac.uk} 

\noindent {\textbf{Tasmin L. Symons}, Mathematics Department, Imperial College London, UK, SW7 2AZ }\\
\noindent \url{tasmin.symons11@imperial.ac.uk} 

\vspace{1em}

\noindent The second author is supported by the EPSRC Centre for Doctoral Training in the Mathematics of Planet Earth.
}
\end{document}